\documentclass[12pt,reqno]{amsart}
\usepackage{amssymb}
\usepackage{times}
\newtheorem{thm}{Theorem}[section]

\newtheorem{prop}[thm]{Proposition}

\theoremstyle{definition}

\newtheorem{eg}[thm]{Example}

\theoremstyle{remark}
\newtheorem*{rmk*}{Remark}

\newcommand{\la}{\lambda}

\DeclareMathOperator{\rank}{rank}

\DeclareMathOperator{\SL}{SL}

\newcommand{\g}{\mathfrak{g}}

\newcommand{\BU}{{\mathbb U}}
\newcommand{\C}{\mathbb{C}}
\newcommand{\ZZ}{\mathbb{Z}}
\newcommand{\QQ}{\mathbb{Q}}

\newcommand{\FF}{\mathbb{F}}

\renewcommand{\sl}{\mathfrak{sl}}

\newcommand{\hsr}{\alpha_0}     

\usepackage{color}

\usepackage{hyperref}
\hypersetup{
  colorlinks=true, 
   linkcolor=red,  
}
\usepackage[hyphenbreaks]{breakurl}

\begin{document}

\title{Globally Irreducible Weyl modules for Quantum Groups}

\author[S. Garibaldi]{Skip Garibaldi}
\address{Center for Communications Research, San Diego, California 92121}
\email{skip@member.ams.org}

\author[R.M. Guralnick]{Robert M. Guralnick}
\address{Department of Mathematics, University of Southern California,
Los Angeles, CA 90089-2532}
\email{guralnic@usc.edu}

\author[D.K. Nakano]{Daniel K. Nakano}
\address{Department of Mathematics, University of Georgia, Athens, Georgia 30602, USA}
\email{nakano@math.uga.edu}

\dedicatory{}

\thanks{The second author was partially supported by NSF grant DMS-1302886 and DMS-1600056.  Research of the third author was partially supported by NSF grant DMS-1402271.}


\begin{abstract} The authors proved that  a Weyl module for a simple algebraic group is irreducible over every field if and only if 
the module is isomorphic to the adjoint representation for $E_{8}$ or its highest weight is minuscule. In this paper, we prove an analogous criteria 
for irreducibility of Weyl modules over the quantum group $U_{\zeta}({\g})$ where ${\g}$ is a complex simple Lie algebra and $\zeta$ ranges over 
roots of unity. 
\end{abstract}

\maketitle

\section{Introduction} Let $G$ be a simple algebraic group over an algebraically closed field $k$. It is well known that Weyl modules of minuscule highest weight 
are irreducible over every field $k$. Gross observed that this also true for the adjoint module for $E_{8}$ and conjectured these are the only cases of Weyl modules 
that are globally irreducible. The authors recently proved Gross' Conjecture in \cite{GGN}; see \cite{Jan:max} for another argument.

Let ${\g}$ be a complex simple Lie algebra and $U_{\zeta}({\g})$ be the quantum group obtained by taking Lusztig's ${\mathcal A}$-form and specializing 
to a root unity. The algebra $U_{\zeta}({\g})$ plays an analogous role to the distribution algebra of a simple algebraic group. Weyl modules can be defined 
at the ${\mathcal A}$-form level and one can ask when they remain irreducible upon specialization for all roots of unity; when that occurs we say that the Weyl module is \emph{globally irreducible}.  The main purpose of our paper is to 
determine which Weyl modules are globally irreducible for quantum groups.  

\begin{thm} \label{MT}
Let ${\g}$ be a complex simple Lie algebra.  The quantum Weyl module $\Delta_{\zeta}(\la)$ is irreducible over $U_{\zeta}({\g})$ for every root of unity $\zeta \in \C^\times$ if and only if 
	\begin{enumerate}
	\renewcommand{\theenumi}{\alph{enumi}}
	\item  \label{MT.min} $\lambda$ is a minuscule dominant weight, or 
	\item \label{MT.E8} $\g$ is of type $E_8$ and $\lambda$ is the highest root $\alpha_0$.
	\end{enumerate}
\end{thm} 

Many of the ideas from \cite{GGN} will be used to reduce the proof of Theorem~\ref{MT} to finitely many cases. Several of the proofs for these cases 
in the algebraic group case do not carry over to the quantum group situation. These include the proof of the global irreducibility of the adjoint module for $E_{8}$ and 
also the conditions for the reducibility of Weyl module with highest weight $\omega_{1}+\omega_{n}$ in type $B_{n}$. We provide suitable replacement proofs involving the use of 
a matrices with quantum entries and translation functors which are of independent interest. 

Even though the statements of the Theorem~\ref{MT} for the algebraic group and quantum group situation are analogous, the underlying result is not identical. 
For example, in the quantum group case there is no lower bound on ${\ell}$ for reducibility of quantum Weyl modules as in the algebraic group case, see \S\ref{sl2.sec}.  We also remark that it is not known how to directly pass information about decomposition numbers unless the characteristic of the field is very large, in which case the Lusztig 
Character Formula holds in both settings.

\subsection*{Acknowledgements} The authors thank Henning Andersen and George Lusztig for their suggestion to extend our prior work \cite{GGN} to the quantum case.

\section{Definitions and notation}  \label{defs}

\subsection{Roots and Weights:} Let $\Phi$ be a finite root system \cite{Hum}, and let 
$\Delta=\{\alpha_1,\cdots,\alpha_n\}$ be a base of simple roots (labeled in the standard Bourbaki way, as in Table \ref{dynks.table}). 
Moreover, let $\Phi^+$ (respectively, $\Phi^-$) be the corresponding set of positive (respectively, negative)
roots. The ${\mathbb R}$-spans of the roots is a Euclidean space $\mathbb E$ with positive definite inner product
$\langle u,v\rangle$, $u,v\in \mathbb E$, adjusted so that $\langle\alpha,\alpha\rangle=2$ if $\alpha\in\Phi$ is a short root.

For $\alpha\in\Phi$,  set $\alpha^\vee=\frac{2}{\langle\alpha,\alpha\rangle}\alpha$ be the corresponding coroot. 
Denote the short root of maximal height in $\Phi$ by $\alpha_0$; thus, $\alpha_0^\vee$ is the
unique long root of maximal length in the dual root system $\Phi^\vee$.  Define the fundamental dominant weights 
$\omega_1,\cdots,\omega_n\in X_+$by the condition that $\langle\omega_i,\alpha_j^\vee\rangle=\delta_{i,j}$, for $1\leq i,j\leq n$.
The Coxeter number of $\Phi$ is defined  to be $h=\langle \rho,\alpha_{0}^{\vee} \rangle +1={\text{ht}}(\alpha_0^\vee) +1$ where $\rho$ is the half sum of 
positive roots. Note that $h-1$ is the height of the maximal root in $\Phi$. Let $W$ be the Weyl group corresponding to $\Phi$, and for $l$ a fixed positive integer, let 
$W_\ell\cong W\ltimes \ell{\mathbb Z}\Phi$ be the affine Weyl group.

Let $X:={\mathbb Z}\omega_1\oplus \cdots\oplus{\mathbb Z}\omega_n$ be the weight lattice, and $X^+
:={\mathbb N}\omega_1\oplus\cdots\oplus {\mathbb N}\omega_n$. The weight lattice $X$ is partially ordered by putting $\lambda\geq
\mu$ if and only if $\lambda-\mu = \sum c_i \alpha_i$ where $c_i$ is a nonnegative integer for all $i$.  The weights in $X^+$ that are minimal with respect to the partial ordering are \emph{minuscule} weights.  
Note that the zero weight is minuscule by this definition (in some references this is not the case).  Every nonzero minuscule weight is a fundamental dominant weight (one of the $\omega_i$'s). 
These are indicated in Table \ref{dynks.table}.  
\begin{table}[bth]
{\centering\noindent\makebox[450pt]{
\begin{tabular}[c]{p{2.2in}|p{2.2in}}

$\small{(A_n)~~}$
\begin{picture}(7,2)(0,0)
\multiput(0,1)(20,0){3}{\circle{6}}
\multiput(62,1)(20,0){3}{\circle{6}}
\put(0,1){\circle*{3}}
\put(0,1){\line(1,0){20}}
\put(20,1){\circle*{3}}
\put(20,1){\line(1,0){20}}
\put(40,1){\circle*{3}}
\put(40,-1.6){ \mbox{$\cdots$}}
\put(62,1){\circle*{3}}
\put(62,1){\line(1,0){20}}
\put(82,1){\circle*{3}}
\put(82,1){\line(1,0){20}}
\put(102,1){\circle*{3}}

\put(-2,-7){\mbox{\tiny $1$}}
\put(18,-7){\mbox{\tiny $2$}}
\put(38,-7){\mbox{\tiny $3$}}
\put(54,-7){\mbox{\tiny $n$$-$$2$}}
\put(75,-7){\mbox{\tiny $n$$-$$1$}}
\put(100,-7){\mbox{\tiny $n$}}
\end{picture}
\vspace{0.5cm}

&

$\small{(E_6)~~}$
\begin{picture}(7,2)(0,0)
\put(0,-5){\circle*{3}}
\put(0,-5){\circle{6}}
\put(0,-5){\line(1,0){15}}
\put(15,-5){\circle*{3}}
\put(15,-5){\line(1,0){15}}
\put(30,-5){\circle*{3}}
\put(30,10){\circle*{3}}
\put(30,-5){\line(0,1){15}}
\put(30,-5){\line(1,0){15}}
\put(45,-5){\circle*{3}}
\put(45,-5){\line(1,0){15}}
\put(60,-5){\circle*{3}}
\put(60,-5){\circle{6}}

\put(-2,-13){\mbox{\tiny $1$}}
\put(13,-13){\mbox{\tiny $3$}}
\put(28,-13){\mbox{\tiny $4$}}
\put(43,-13){\mbox{\tiny $5$}}
\put(58.5,-13){\mbox{\tiny $6$}}
\put(33,9){\mbox{\tiny $2$}}
\put(22,9){\mbox{\tiny $\star$}}

\end{picture}

\\

$\small{(B_n)~~}$
\begin{picture}(7,2)(0,0)
\put(0,1){\circle*{3}}
\put(0,1){\line(1,0){20}}
\put(20,1){\circle*{3}}
\put(20,1){\line(1,0){20}}
\put(40,1){\circle*{3}}
\put(40,-1.6){ \mbox{$\cdots$}}
\put(62,1){\circle*{3}}
\put(62,1){\line(1,0){20}}
\put(82,1){\circle*{3}}
\put(82,2){\line(1,0){20}}
\put(82,0){\line(1,0){20}}
\put(89,-1){{\tiny\mbox{$>$}}}
\put(102,1){\circle*{3}}
\put(102,1){\circle{6}}

\put(-2,-7){\mbox{\tiny $1$}}
\put(-2,5){\mbox{\tiny $\star$}}
\put(18,-7){\mbox{\tiny $2$}}
\put(38,-7){\mbox{\tiny $3$}}
\put(54,-7){\mbox{\tiny $n$$-$$2$}}
\put(75,-7){\mbox{\tiny $n$$-$$1$}}
\put(100,-7){\mbox{\tiny $n$}}
\end{picture}
\vspace{0.5cm}

&

$\small{(E_7)~~}$
\begin{picture}(7,2)(0,0)
\put(0,-5){\circle*{3}}
\put(0,-5){\circle{6}}
\put(0,-5){\line(1,0){15}}
\put(15,-5){\circle*{3}}
\put(15,-5){\line(1,0){15}}
\put(30,-5){\circle*{3}}
\put(30,-5){\line(1,0){15}}
\put(45,-5){\circle*{3}}
\put(45,-5){\line(1,0){15}}
\put(45,10){\circle*{3}}
\put(45,-5){\line(0,1){15}}
\put(60,-5){\circle*{3}}
\put(60,-5){\line(1,0){15}}
\put(75,-5){\circle*{3}}

\put(-2,-13){\mbox{\tiny $7$}}
\put(13,-13){\mbox{\tiny $6$}}
\put(28,-13){\mbox{\tiny $5$}}
\put(43,-13){\mbox{\tiny $4$}}
\put(58.5,-13){\mbox{\tiny $3$}}
\put(73,-13){\mbox{\tiny $1$}}
\put(73,0){\mbox{\tiny $\star$}}

\put(47.5,9){\mbox{\tiny $2$}}
\end{picture}

\\

$\small{(C_n)~~}$
\begin{picture}(7,2)(0,0)
\put(0,1){\circle*{3}}
\put(0,1){\line(1,0){20}}
\put(20,1){\circle*{3}}
\put(20,1){\line(1,0){20}}
\put(40,1){\circle*{3}}
\put(40,-1.6){ \mbox{$\cdots$}}
\put(62,1){\circle*{3}}
\put(62,1){\line(1,0){20}}
\put(82,1){\circle*{3}}
\put(82,2){\line(1,0){20}}
\put(82,0){\line(1,0){20}}
\put(89,-1){{\tiny\mbox{$<$}}}
\put(102,1){\circle*{3}}
\put(0,1){\circle{6}}

\put(-2,-7){\mbox{\tiny $1$}}
\put(18,-7){\mbox{\tiny $2$}}
\put(18,5){\mbox{\tiny $\star$}}

\put(38,-7){\mbox{\tiny $3$}}
\put(54,-7){\mbox{\tiny $n$$-$$2$}}
\put(75,-7){\mbox{\tiny $n$$-$$1$}}
\put(100,-7){\mbox{\tiny $n$}}
\end{picture}
\vspace{0.5cm}

&

$\small{(E_8)~~}$
\begin{picture}(7,2)(0,0)
\put(0,-5){\circle*{3}}
\put(0,-5){\line(1,0){15}}
\put(15,-5){\circle*{3}}
\put(15,-5){\line(1,0){15}}
\put(30,-5){\circle*{3}}
\put(30,-5){\line(1,0){15}}
\put(45,-5){\circle*{3}}
\put(60,-5){\line(0,1){15}}
\put(60,-5){\circle*{3}}
\put(60,10){\circle*{3}}
\put(75,-5){\circle*{3}}
\put(75,-5){\line(1,0){15}}
\put(90,-5){\circle*{3}}
\put(45,-5){\line(1,0){15}}
\put(60,-5){\line(1,0){15}}

\put(-2,-13){\mbox{\tiny $8$}}
\put(-2,0){\mbox{\tiny $\star$}}

\put(13,-13){\mbox{\tiny $7$}}
\put(28,-13){\mbox{\tiny $6$}}
\put(43,-13){\mbox{\tiny $5$}}
\put(58.5,-13){\mbox{\tiny $4$}}
\put(73,-13){\mbox{\tiny $3$}}
\put(88,-13){\mbox{\tiny $1$}}
\put(62.5,9){\mbox{\tiny $2$}}
\end{picture}

\\

$\small{(D_n)~~}$
\begin{picture}(7,2)(0,0)
\put(0,1){\circle*{3}}
\put(0,1){\circle{6}}
\put(0,1){\line(1,0){20}}
\put(20,1){\circle*{3}}
\put(20,1){\line(1,0){20}}
\put(40,1){\circle*{3}}
\put(40,-1.6){ \mbox{$\cdots$}}
\put(62,1){\circle*{3}}
\put(62,1){\line(1,0){20}}
\put(82,1){\circle*{3}}
\put(82,2){\line(4,3){15}}
\put(82,0){\line(4,-3){15}}
\put(96.5,12.9){\circle*{3}}
\put(96.5,12.9){\circle{6}}
\put(96.5,-10.9){\circle*{3}}
\put(96.5,-10.9){\circle{6}}

\put(-2,-7){\mbox{\tiny $1$}}

\put(18,5){\mbox{\tiny $\star$}}
\put(18,-7){\mbox{\tiny $2$}}
\put(38,-7){\mbox{\tiny $3$}}
\put(54,-7){\mbox{\tiny $n$$-$$3$}}
\put(86,-0.5){\mbox{\tiny $n$$-$$2$}}
\put(100,-12){\mbox{\tiny $n$}}
\put(100,11.8){\mbox{\tiny $n$$-$$1$}}
\end{picture}

&

$\small{(F_4)~~}$
\begin{picture}(7,2)(0,0)
\put(0,1){\circle*{3}}
\put(0,1){\line(1,0){15}}
\put(15,1){\circle*{3}}
\put(15,0){\line(1,0){15}}
\put(15,2){\line(1,0){15}}
\put(19,-1){{\tiny\mbox{$>$}}}
\put(30,1){\circle*{3}}
\put(30,1){\line(1,0){15}}
\put(45,1){\circle*{3}}

\put(-2,-7){\mbox{\tiny $1$}}
\put(13,-7){\mbox{\tiny $2$}}
\put(28,-7){\mbox{\tiny $3$}}
\put(43,-7){\mbox{\tiny $4$}}
\put(43,5){\mbox{\tiny $\star$}}

\put(65,1){$\small{(G_2)~~}$}
\put(92,1){\circle*{3}}
\put(92,0.1){\line(1,0){15}}
\put(92,1.1){\line(1,0){15}}
\put(92,2.1){\line(1,0){15}}
\put(96,-1){{\tiny\mbox{$<$}}}
\put(107,1){\circle*{3}}

\put(90,-7){\mbox{\tiny $1$}}
\put(90,5){\mbox{\tiny $\star$}}

\put(105,-7){\mbox{\tiny $2$}}
\end{picture}

\end{tabular}
}}
\vskip .5cm
\caption{Dynkin diagrams of simple root systems, with simple roots numbered.  A circle around vertex $i$ indicates that the fundamental weight $\omega_i$ is minuscule.  A $\star$ indicates that $\omega_i$ is the highest short root $\hsr$.  The highest short root of $A_n$ is $\omega_1 + \omega_n$.} \label{dynks.table}
\end{table}

\subsection{Quantum Groups:} Let ${\g}$ be a complex simple Lie algebra with associated irreducible root 
system $\Phi$. The goal of this section is to define the quantum enveloping algebra $\BU_q({\g})$ and an ${\mathcal A}$-form involving 
divided powers which is contained in $\BU_q({\g})$ that can be specialized to any primitive ${\ell}$-th root of unity. Sawin provides a uniform treatment in  \cite{Saw}. 

Let ${\mathcal A}={\mathbb Q}[q,q^{-1}]$ be the $\mathbb Q$-algebra of Laurent
polynomials in an indeterminate $q$ with fraction field ${\mathbb Q}(q)$. The quantum enveloping algebra $\BU_q({\g})$ is the ${\mathbb Q}(q)$-algebra with generators
$E_{\alpha}$, $K_\alpha$, $K_\alpha^{-1}$, and $F_\alpha$ for $\alpha\in \Delta$, subject to the relation
\renewcommand{\theequation}{R\arabic{equation}}
\begin{eqnarray}
K_\alpha K_\alpha^{-1} = 1 \text{\ and\ } K_\alpha K_\beta = K_\beta K_\alpha.
\end{eqnarray}
We further set $d_\alpha = \langle{\alpha, \alpha}\rangle/2$ for $\alpha \in \Phi$ and $q_\alpha = q^{d_\alpha}$ and impose the relations
\begin{eqnarray}
K_\alpha E_\beta K_\alpha^{-1}=q_\alpha^{\langle\beta,\alpha^\vee\rangle} E_\beta=q^{\langle\beta,\alpha\rangle}E_\beta;\\
K_\alpha F_\beta K_\alpha^{-1}=q_\alpha^{-\langle\beta,\alpha^\vee\rangle} F_\beta=q^{-\langle\beta,\alpha\rangle}F_\beta
\end{eqnarray}%
\setcounter{equation}{0}%
\renewcommand{\theequation}{\arabic{equation}}%
and further relations (R4), (R5), (R6) for which we refer to \cite[4.3]{Jan2}.  We remark that the algebra $\BU_q({\g})$ is a Hopf algebra (cf.~\cite{BNPP}) 

The quantum enveloping algebra $\BU_q({\g})$ has a natural $\mathcal A$-form, $\BU_q^{\mathcal A}({\g})$ due to Lusztig. That is, 
${\mathbb U}_q^{\mathcal A}({\g})$ is an $\mathcal A$-subalgebra of ${\mathbb U}_q({\g})$ that is free 
as an $\mathcal A$-module with 
$${\mathbb U}_q^{\mathcal A}({\g})\otimes_{\mathcal A}{\mathbb Q}(q) \cong {\mathbb U}_q({\g}).$$
The construction of this $\mathcal A$-form is described below.

For an integer $i$, put 
\begin{equation} \label{ibrack}
[i]_{q}= \frac{q^i - q^{-i}}{q - q^{-1}},
\end{equation}
and set, for $i>0$, $[i]_{q}^! =[i]_{q} [i-1]_{q}\cdots [1]_{q}$. By
convention, $[0]_{q}^!=1$. For any integer $n$ and positive integer $m$, write
$$\left[\begin{matrix} n \\ m\end{matrix}\right]_{q}=\frac{[n]_{q}[n-1]_{q}\cdots [n-m+1]_{q}}{[1]_{q}[2]_{q}\cdots[m]_{q}}.$$
Set $\left[\begin{smallmatrix} n\\ 0\end{smallmatrix}\right]_{q}=1$, by definition. The expressions $[i]_{q}$ and $\left[\begin{smallmatrix} n \\ m\end{smallmatrix}\right]_{q}$
 all belong to
$\mathcal A$ (in fact, they belong to ${\mathbb Z}[q,q^{-1}]$). In case the 
root system has two root lengths, some
scaling of the variable $q$ is required. Thus, given any Laurent polynomial
$f\in\mathcal A$ and $\alpha\in\Pi$, let $f_\alpha\in\mathcal A$ be obtained by replacing $q$ throughout by
$q_\alpha$.

For $\alpha\in\Pi$ and $m\geq 0$, let
$$\begin{cases} E^{(m)}_\alpha=
\frac{E_\alpha^m}{[m]^!_\alpha}\in{\mathbb U}_q({\mathfrak g}) \\
F^{(m)}_\alpha= \frac{F_\alpha^m}{[m]^!_\alpha}\in\BU_q({\mathfrak g}).\end{cases}
$$
be the $m$-th ``divided powers."    Let
$$\BU_q^{\mathcal A}({\g}):=\left\langle E^{(m)}_\alpha,\, F^{(m)}_\alpha,\, K_\alpha^{\pm 1}\,|\, \alpha\in\Pi, m\in{\mathbb N}\right\rangle\subset\BU_q,$$
where $\langle \cdots\rangle$ means ``$\mathcal A$-subalgebra generated by."

For any $\epsilon\in {\mathbb C}^\times$, set
\[
[i]_{\epsilon}=\lim_{q\rightarrow \epsilon}\, [i]_{q}=\lim_{q\rightarrow \epsilon} \frac{q^{2i}-1}{q^{2}-1}\cdot q^{i-1}.
\]
The other definitions defined above that involve $[i]_{q}$ can also be specialized to $\epsilon$. 

Suppose that ${\ell }\geq 2$ and $\zeta\in \C^\times$ has (finite) order $\ell$. 
Throughout this paper, we will use the following properties: 
\begin{itemize} 
\item If ${\ell}>2$ and ${\ell}\mid i$ then $[i]_{\zeta}=0$. 
\item If $\zeta$ is a primitive $4$th root of unity then $[2]_{\zeta}=0$. 
\item If $\zeta=\pm 1$ then $[i]_{\zeta}\neq 0$. 
\end{itemize} 

Set $k={\mathbb Q}(\zeta)\subset\mathbb C$ which will be regarded as an $\mathcal
A$-algebra via the homomorphism $\mathbb{Q}[q,q^{-1}]\to k$ defined by $q\mapsto \zeta$. 
Set 
\begin{equation*}\label{firstquantumgroup}
U_\zeta({\g}):=\BU_{q}^{\mathcal A}({\mathfrak g}) \otimes_{\mathcal A} {\mathbb C},
\end{equation*}
where $\ell$ is the order of $\zeta$ in $\C^\times$. Here $\mathbb C$ is regarded as an $\mathcal A$-algebra via the algebra homomorphism
${\mathcal A}\to{\mathbb C}$ defined by $q\mapsto \zeta$.  The Hopf algebra structure on
$\BU_q({\g})$ induces a Hopf algebra structure on $\BU^{\mathcal A}_q({\g})$. From the passage to the field
$\mathbb C$, one obtains a Hopf algebra structure on the algebra $U_\zeta({\g})$. 

In this paper we will consider only $U_\zeta({\g})$-modules which are integrable and type 1. 
In particular, any such $M$ decomposes into a direct sum $\bigoplus_{\lambda\in X}M_\lambda$ of $M_\lambda$
weight spaces for $\lambda\in X$, and each $E_i,F_j$ acts locally nilpotently on $M$. On the weight spaces, one 
has for $v\in M_\lambda$,
\begin{equation} K_\alpha v=\zeta^{\langle\lambda,\alpha\rangle}v;
\end{equation}
\begin{equation}
\left[\begin{smallmatrix}
K_\alpha; m\\
n\end{smallmatrix}\right]v=\left[\begin{smallmatrix}\langle
\lambda,\alpha\rangle +m\\
n\end{smallmatrix}\right]_{q=\zeta^{d_\alpha}}v
\end{equation}
for all $\alpha\in\Delta, m\in{\mathbb Z}, n\in{\mathbb N}$. For the definition of $\left[\begin{smallmatrix}
K_\alpha; m\\n\end{smallmatrix}\right]$, see \cite[\S 2.2]{BNPP}. 

\subsection{Induced and Weyl modules:}
For $\lambda\in X_{+}$, let 
$$\nabla(\lambda):=\nabla_{\zeta}(\lambda)=\text{ind}_{U_{\zeta}({\mathfrak b})}^{U_{\zeta}({\g})}\lambda$$ 
be the (quantum) induced module whose 
character is given by Weyl's character formula, and $\Delta(\lambda)=\Delta_{\zeta}(\lambda) = \nabla_{\zeta}(-w_{0}\lambda)^{*}$ be the quantum Weyl module, compare \cite[\S6]{Lusztig}.  These are the modules considered in Theorem \ref{MT}.

In the special case where $\zeta = 1$, the definition \eqref{ibrack} becomes $[i]_{\zeta^{d}} = i$ and we find $[m]^!_\alpha = m!$ and $\left[ \begin{smallmatrix} n \\ m \end{smallmatrix} \right]_\alpha = \binom{n}{m}$, so $U_\zeta(\g)$ is the usual universal enveloping algebra of $\g$ over ${\mathbb C}$.  In that case, we find that the Weyl module $\Delta_\zeta(\lambda)$ is the irreducible module of $\g$ of highest weight $\la$.

\section{Example: Weyl modules of $U_\zeta(\sl_2)$} \label{sl2.sec}

Consider $U_{\zeta}(\sl_{2})$ for $\zeta \in \C^\times$ of order $\ell$.  We may identify the dominant weights with the set $\ZZ_+$ of non-negative integers.  For $j \ge 0$ define
\[
s_{j}:=\begin{cases}
 j & \text{if $j$ is odd;} \\
j/2 
&\text{if $j$ even.}
\end{cases}
\]

\begin{prop} \label{thm:sl2quantum} 
The quantum Weyl module $\Delta(\lambda)$ for $U_{\zeta}(\sl_{2})$ is irreducible if and only if one of the following holds:
\renewcommand{\theenumi}{\alph{enumi}} 
\begin{enumerate} 
\item \label{sl2.rest} $0\leq \lambda < s_\ell$.
\item \label{sl2.nonrest} $\lambda \equiv - 1 \pmod {s_\ell}$.
\end{enumerate} 
\end{prop} 

\begin{proof}
For odd ${\ell} \ge 3$, the result can be deduced by  Steinberg's tensor product theorem for quantum groups, see
\cite[Prop.~9.2]{Lusztig}.  
In the trivial case when $\ell = 1$, $\Delta(\lambda)$ is irreducible for all $\lambda$, which confirms the claim.

For even ${\ell }$, one can use the explicit generators and relations of the dual of $\Delta(\lambda)$ from 
\cite[5A.7]{Jan2}. The dual has basis $\{v_{0},v_{1},\dots,v_{\lambda}\}$ with 
$$E_{\alpha}^{(m)}.v_{j}=\left[\begin{smallmatrix} j+m\\ m\end{smallmatrix}\right]_{\zeta} v_{j+m}.$$ 
The conditions \eqref{sl2.rest} and \eqref{sl2.nonrest} are equivalent to showing that there are no non-trivial maximal vectors which can be deduced by 
analyzing the aforementioned formula, compare \cite{Cl}. 
\end{proof}

\begin{eg} \label{sl2.eg}
$\Delta(0)$ and $\Delta(1)$ are irreducible for $U_\zeta(\sl_2)$ for all roots of unity $\zeta$.  Compare this to $\Delta(2)$, which is reducible if and only if $\zeta$ has order 4.  More generally, for each $\lambda \ge 2$, there is some $\ell$ with $s_\ell = \lambda$, and $\Delta(\lambda)$ is reducible for $U_\zeta(\sl_2)$ where $\zeta$ has order $\ell$.
\end{eg}

\begin{eg} \label{S:nobound}
Pick some $t \ge 4$ and set $\lambda :=   s_1 s_2 \cdots s_t - 1$.  For each $\ell = 1, \ldots, t$, we have $\lambda > s_\ell$ and $s_1 s_2 \cdots s_{\ell-1} s_{\ell+1} \cdots s_t \ge 1$, whence \ref{thm:sl2quantum}\eqref{sl2.nonrest} holds and $\Delta(\lambda)$ is irreducible for $U_\zeta(\sl_2)$ where $\zeta \in \C^\times$ has order $\ell$.
\end{eg}

In the analogue of Theorem \ref{MT} for a simple algebraic group $G$, \cite[Th.~1.1]{GGN}, it was shown that for every dominant weight $\lambda$ of $G$, there is a prime $\ell \le 2(\rank G) + 1$ such that the Weyl module $\Delta(\lambda) \otimes \FF_\ell$ is reducible.  In particular, in case $G = \SL_2$, $\Delta(\lambda) \otimes \FF_\ell$ is reducible for $\ell = 2$ or 3.  Example \ref{S:nobound} above shows that no such bound exists in the setting of quantum groups.  One might view the reason for this difference as being that one can only iterate the Frobenius once in the quantum case.

\section{Levi subalgebras}
\subsection{Levi subalgebra and Parabolics:} Lusztig has defined an algebra automorphism
$T_{\alpha}:\BU_q({\g})\to\BU_q({\g})$. By using this automorphism, one can 
construct a PBW type basis for quantum groups by defining root vectors for general $\alpha\in \Phi$
(cf.~\cite[Ch. 8]{Jan2}). 

If $s = s_{\alpha} \in W$ is the simple reflection defined by $\alpha$,  set $T_{s} := T_{\alpha}$.
Given any $w \in W$, let $w = s_{\beta_1}s_{\beta_2}\cdots s_{\beta_n}$ be
a reduced expression. Define $T_w := T_{\beta_1}\cdots T_{\beta_n}
\in \text{Aut}(\BU_q({\g}))$. 

Now let $J \subseteq \Delta$ and fix a reduced expression $w_0 =
s_{\beta_1}\cdots s_{\beta_N}$ that begins with a reduced expression for the element
long element of $w_{0,J}$ the Weyl group for the Levi subgroup $L_{J}$. 
If $w_{0,J} =s_{\beta_1}\cdots s_{\beta_M}$, then $s_{\beta_{M+1}}\cdots s_{\beta_N}$ is a
reduced expression for $w_J=w_{0,J}w_0$. Now there exists a linear ordering $\gamma_1
\prec \gamma_2 \prec \cdots \prec \gamma_N$ of the positive roots, where
$\gamma_i = s_{\beta_1}\cdots s_{\beta_{i-1}}(\beta_i)$. For $\gamma = \gamma_i \in
\Phi^+$, the ``root vector'' $E_{\gamma} \in \BU_q({\g})$ is defined by
$$
E_{\gamma} = E_{\gamma_i} := T_{s_{\beta_1} \cdots s_{\beta_{i-1}}}(E_{\beta_i})
=
        T_{\beta_1}\cdots T_{\beta_{i-1}}(E_{\beta_i}).
$$
Furthermore, $E_{\gamma}$ has weight $\gamma$.
Similarly,
$$
F_{\gamma} = F_{\gamma_i} := T_{s_{\beta_1} \cdots s_{\beta_{i-1}}}(F_{\beta_i})
=
        T_{\beta_1}\cdots T_{\beta_{i-1}}(F_{\beta_i}),
$$
a root vector of weight $-\gamma$. If $\gamma\in \Delta$ then $E_{\gamma}$ coincides with the original
generator. 

Let $J \subseteq \Pi$ and consider the Levi and parabolic Lie subalgebras
${\mathfrak l}_J$ and ${\mathfrak p}_J = {\mathfrak l}_J\oplus{\mathfrak u}_J$ of $\g$.  
We can define corresponding quantum
enveloping algebras $\BU_q({\mathfrak l}_{J})$ and $\BU_q({\mathfrak p}_{J})$. As Hopf subalgebras of
$\BU_q({\g})$, 
$$\BU_q({\mathfrak l}_{J})=\langle  \{E_{\alpha},F_{\alpha} : \alpha\in J\} \cup \{K_{\alpha}^{\pm 1} : \alpha \in \Delta\} \rangle$$
and
$$\BU_q({\mathfrak p}_{J})=\langle \{E_{\alpha}: \alpha \in J\}
\cup \{F_{\alpha}, K_{\alpha}^{\pm 1} : \alpha \in \Delta\} \rangle .$$ 

In the case when $J =\varnothing$, then ${\mathfrak l}_{J}= \mathfrak{h}$, ${\mathfrak p}_{J} = \mathfrak{b}$,
Upon specialization we obtain the subalgebras
$U_{\zeta}({\mathfrak l}_{J})$, $U_{\zeta}({\mathfrak p}_{J})$, $U_{\zeta}({\mathfrak h})$, and 
$U_{\zeta}({\mathfrak b})$. 

\subsection{Restriction to Levi subalgebras:} 
For $J\subseteq \Delta$, set 
$$X_{J}^{+}:=\{\lambda\in X:\ 0\leq \langle \lambda,\alpha^{\vee} \rangle \ \text{for all $\alpha\in J$}\}.$$ 
If $\lambda\in X_{J}^{+}$, one can define the induced module $\Delta_{J}(\lambda)$ with simple $U_{\zeta}({\mathfrak l}_{J})$-socle 
$L_{J}(\lambda)$ dually a Weyl module $\Delta_{J}(\lambda)$ with head $L_{J}(\lambda)$. 

\begin{thm}\label{thm:Levireduction} Let ${\g}$ be a simple complex Lie algebra and $\zeta \in \C^\times$ be a root of unity. 
If $\Delta(\lambda)$ is an irreducible $U_{\zeta}({\g})$-module then $\Delta_J(\lambda)$ is an 
irreducible $U_{\zeta}({\mathfrak l}_J)$-module for any $J\subseteq \Delta$. 
\end{thm}

\begin{proof} The argument will follow the line of reasoning given in \cite{GGN} with some modifications. 
By using the argument in \cite[II 5.21]{Jan}, there exists a weight space decomposition for $\nabla(\lambda)$ 
given by 
$$
\nabla(\lambda)=\left(\bigoplus_{\nu\in 
{\mathbb Z}J} \nabla(\lambda)_{\lambda-\nu}\right) \oplus M.
$$ 
where $M$ is the direct sum of all weight spaces $\nabla(\lambda)_{\sigma}$ with 
$\sigma\neq \lambda-\nu$ for any $\nu\in {\mathbb Z}J$. Moreover, $\nabla_{J}(\lambda)=\oplus_{\nu\in {\mathbb Z}J} \nabla(\lambda)_{\lambda-\nu}$ with  
the aforementioned decomposition being stable under the action of $U_{\zeta}({\mathfrak l}_{J})$. Consequently, 
as $U_{\zeta}({\mathfrak l}_{J})$-modules: 
\begin{equation} \label{eq:induceddec}
\nabla(\lambda)\cong \nabla_{J}(\lambda)\oplus M. 
\end{equation} 
One can also apply a dual argument for Weyl modules to get a 
decomposition as $U_{\zeta}({\mathfrak l}_{J})$-modules:
\begin{equation} \label{eq;Weyldecomp}
\Delta(\lambda)\cong \Delta_{J}(\lambda)\oplus M^{\prime}. 
\end{equation} 
for some $U_{\zeta}({\mathfrak l}_{J})$-module $M^{\prime}$. 

One has $L(\lambda)=\text{soc}_{U_{\zeta}({\g})}(\nabla(\lambda))$, thus    
$\text{soc}_{U_{\zeta}({\mathfrak l}_{J})}L(\lambda)\subseteq \text{soc}_{U_{\zeta}({\mathfrak l}_{J})}(\nabla(\lambda))$. 
Observe that 
\begin{equation}
L_{J}(\lambda)=\text{soc}_{U_{\zeta}({\mathfrak l}_{J})}(\nabla_{J}(\lambda))\subseteq 
\text{soc}_{U_{\zeta}({\mathfrak l}_{J})}(\nabla(\lambda)).
\end{equation}
The irreducible representation $L_{J}(\lambda)$ appears as an
$U_{\zeta}({\mathfrak l}_{J})$-composition factor of $L(\lambda)$ and 
$\nabla(\lambda)$ with multiplicity one. One can conclude that $L_{J}(\lambda)$ must occur 
in $\text{soc}_{U_{\zeta}({\mathfrak l}_{J})} L(\lambda)$. One can use a similar argument to deduce that 
$L_{J}(\lambda)$ appears in the head of $L(\lambda)$ as $U_{\zeta}({\mathfrak l}_{J})$-module. Since $L_{J}(\lambda)$ has multiplicity one in $L(\lambda)$ now shows that there is an $U_{\zeta}({\mathfrak l}_{J})$-decomposition: 
\begin{equation} \label{eq:simpledec}
L(\lambda)\cong L_{J}(\lambda)\oplus M^{\prime \prime}.
\end{equation}

Now suppose that $\Delta(\lambda)=L(\lambda)$ is irreducible as $U_{\zeta}({\g})$-module. Now one can compare 
the $U_{\zeta}({\mathfrak l}_{J})$-decompositions (\ref{eq;Weyldecomp}) and (\ref{eq:simpledec}) with the facts that 
$L_{J}(\lambda)$ has multiplicity one in $L(\lambda)$ and the indecomposability of $\Delta_{J}(\lambda)$ to conclude that 
$\Delta_{J}(\lambda)=L_{J}(\lambda)$. 
\end{proof}

\section{Analysis of $\Delta_{\zeta}(\alpha_{0})$}  \label{S:shortroothighestweights}

\subsection{} In this section we will analyze $\Delta(\alpha_{0})$ where $\alpha_{0}$ is the highest short root. This module is obtained by 
base change of the Weyl module $\bar{\Delta}(\alpha_{0})$ that is defined over $\BU_q({\g})$. A basis for $\bar{\Delta}(\alpha_{0})$ is 
given in \cite[5.A.2]{Jan2}. Let $\Phi_{s}$ denote the short roots of $\Phi$ and $\Delta_{s}$ be the simple short roots in $\Delta$. The set
$$\{x_{\gamma}:\ \gamma\in \Phi_{s}\}  \cup \{h_{\beta}:\ \beta\in \Delta_{s} \}$$
is a basis for $\Delta(\alpha_0)$.
From the module relations, one can see that this is an ${\mathcal A}$-lattice that is 
stable under the action of $\BU_q^{\mathcal A}({\mathfrak b})$, and coincides with 
$\BU_q^{\mathcal A}({\mathfrak b}).x_{\alpha_{0}}$. In order to obtain $\Delta(\alpha_{0})$ we take this 
${\mathcal A}$-lattice then specialize to $q$ to $\zeta$. 

Using the relations given in \cite[5A.2]{Jan2}, 
\begin{equation*} 
E^{(m)}_\alpha.h_{\beta}=0,\ \ \ F^{(m)}_\alpha.h_{\beta}=0
\end{equation*} 
for all $\alpha,\beta\in \Delta_{s}$ with $m\geq 2$. 
In the case where $m=1$, $\alpha,\beta\in \Delta_{s}$: 
\begin{equation*} 
E_{\alpha}^{(1)}.h_{\beta}=\begin{cases} 
[2]_{\zeta}.x_{\alpha} & \alpha=\beta \\
x_{\alpha} & \alpha\neq \beta,\  \langle \beta,\alpha^{\vee} \rangle =-1 \\
0  & \text{else} \\
\end{cases} 
\end{equation*} 
A similar relation holds when $E_{\alpha}^{(1)}$ is replaced by $F_{\alpha}^{(1)}$. 

Let $\Delta_{s}=\{\beta_{1},\beta_{2},\dots,\beta_{m}\}$. Consider $a_{1}h_{\beta_{1}}+a_{2}h_{\beta_{2}}+\dots+a_{m}h_{\beta_{m}} \in \Delta(\alpha_{0})_{0}$. 
This will be invariant under $E_{\beta_{i}}^{(n)}$ and $F_{\beta_{i}}^{(n)}$ for $n\geq 1$, $i=1,2,\dots,m$  if and only if the matrix $D=(d_{i,j})$ has 
determinant equal to zero where 
\begin{equation*} 
d_{i,j}= \begin{cases} 
[2]_{\zeta} & i=j \\
1   & \langle \beta_{i},\beta_{j}^{\vee} \rangle =-1\\
0   & \text{else} \\
\end{cases} 
\end{equation*} 

In order to show that $\Delta(\alpha_{0})$ is reducible, it suffices to prove that $\Delta(\alpha_{0})$ contains a trivial module in its socle. 
The analysis above implies that this occurs when the determinant of the matrix $D$ is zero. The following theorem analyzes 
when this occurs for each $\Phi$. 

\begin{thm} \label{thm:shortroot} Let $\Delta(\alpha_{0})$ be the quantum Weyl module of highest weight $\alpha_{0}$ over $U_{\zeta}({\g})$ for $\zeta \in \C^\times$ of order $\ell$. Then 
\begin{enumerate} 
\renewcommand{\theenumi}{\alph{enumi}}
\item \label{short.A} If $\Phi=A_{n}$, ${\ell}>2$ and ${\ell}\mid n+1$ then $\Delta(\alpha_{0})$ is reducible. 
\item \label{short.B} If $\Phi=B_{n}$ and ${\ell }=4$  then $\Delta(\alpha_{0})$ is reducible. 
\item \label{short.C} If $\Phi=C_{n}$, $n\geq 3$ and ${\ell}\mid n$ then $\Delta(\alpha_{0})$ is reducible. 
\item \label{short.D} If $\Phi=D_{n}$ and ${\ell}=4$ then $\Delta(\alpha_{0})$ is reducible. 
\item \label{short.F} If $\Phi=F_{4}$ and ${\ell }=3$  then $\Delta(\alpha_{0})$ is reducible. 
\item \label{short.G} If $\Phi=G_{2}$ and ${\ell }=4$  then $\Delta(\alpha_{0})$ is reducible. 
\item \label{short.E6} If $\Phi=E_{6}$ and ${\ell }=3$ then $\Delta(\alpha_{0})$ is reducible. 
\item \label{short.E7} If $\Phi=E_{7}$ and ${\ell }=4$ then $\Delta(\alpha_{0})$ is reducible. 
\item \label{short.E8} If $\Phi=E_{8}$ then $\Delta(\alpha_{0})$ is irreducible for all ${\ell}$. 
\end{enumerate} 
\end{thm} 

\begin{proof} Let $\Phi=A_{n}$. We will show that $\det(D)=[n+1]_{\zeta}$ by using induction on $n$. 
This is clear for $n=1$. Assume that this holds for $n-1$, and consider $\Phi=A_{n}$. Let $\Delta_{s}=\{\alpha_{1},\alpha_{2},\dots,\alpha_{n}\}$ 
be the standard ordering of simple roots. Then by expanding along the first row, one has 
$$\det(D)=[2]_{\zeta}[n]_{\zeta}+(-1)[n-2]_{\zeta}=[n+1]_{\zeta}.$$ 
Consequently, if ${\ell }>2$ and ${\ell}\mid n+1$ then $\Delta(\alpha_{0})$ is reducible.

For $\Phi=B_{n}$, and $G_{2}$ there is only one short root and in this case $\det(D)=[2]_{\zeta}$. In the case of 
$\Phi=C_{n}$, $\Delta_{s}=\{\alpha_{1},\alpha_{2},\dots,\alpha_{n-1}\}$. So we are reduced to type $A_{n-1}$ and 
$\det(D)=[n]_{\zeta}$. For $\Phi=F_{4}$, there are two short roots and one has $\det(D)=[2]_{\zeta}[2]_{\zeta}-(1)(1)=[3]_{\zeta}$. 

In the case when $\Phi=D_{n}$ $n\geq 4$, one first considers the case $\Phi=D_{4}$ where 
$\det(D)=([2]_{\zeta})^{2}([2]_{\zeta}^{2}-3)$, which is zero when ${\ell }=4$. Now by expansion along the first row, one can demonstrate in 
the general case for $\Phi=D_{n}$, 
$$\det(D)=[2]_{\zeta}(\det(D^{\prime}))-\det(D^{\prime\prime})$$ where 
$D^{\prime}$ (resp. $D^{\prime \prime}$ is the matrix for $\Delta(\alpha_{0})$ in the case when $\Phi=D_{n-1}$ (resp. $D_{n-2}$). 
This shows that in general by using the equation above and induction that $D$ is zero for ${\ell}=4$. 

For $\Phi=E_{n}$, one can expand along the second row of the matrix D (with rows and columns under the Bourbaki ordering) and use the 
computation for type $A_{n-1}$ to see that  
\begin{equation*} 
\det(D)=
\begin{cases}   [2]_{\zeta}[6]_{\zeta}-[3]^{2}_{\zeta} & \text{if $\Phi=E_{6}$;} \\
[2]_{\zeta}[7]_{\zeta}-[3]_{\zeta}[4]_{\zeta} & \text{if $\Phi=E_{7}$;} \\
[2]_{\zeta}[8]_{\zeta}-[3]_{\zeta}[5]_{\zeta} & \text{if $\Phi=E_{8}$.} 
\end{cases}
\end{equation*}
From these equations one can see that $\Delta(\alpha_{0})$ contains a trivial module (and is reducible) for $\Phi=E_{6}$ (resp. $\Phi=E_{7}$) when 
${\ell}=3$ (resp. ${\ell }=4$). 

Finally, we want to show that $\Delta(\alpha_{0})$ is irreducible when $\Phi=E_{8}$ for all ${\ell}$. One has $\det(D)=[2]_{\zeta}[8]_{\zeta}-[3]_{\zeta}[5]_{\zeta}$. 
By direct calculation,
\begin{equation*}  
\det(D)=\lim_{q\rightarrow \zeta} ([2]_{q}[8]_{q}-[3]_{q}[5]_{q})=\lim_{q\rightarrow \zeta} \frac{f(q)}{q^{8}(q^{2}-1)}  
\end{equation*} 
where 
$$f(q)=q^{20}-q^{18}-q^{16}+q^{12}+q^{8}-q^{4}-q^{2}+1.$$
One can show directly from the equation above, if $\ell = 2$ (i.e., $\zeta = - 1$), then $\det(D)\neq 0$. 
Furthermore, $\det(D)\neq 0$ if and only if the ${\ell}$-th cyclotomic polynomial $\Phi_{\ell}(q)$ does not divide the polynomial $f(q)$. 

Now, $f(q) = (q-1)^2 (q+1)^2 f_{16}(q)$, where $f_{16}(q) := q^{16} + q^{14} - q^{10} - q^8 - q^6 + q^2 + 1$ is irreducible in $\QQ[q]$.  Since $\Phi_\ell(q)$ is irreducible and $\deg \Phi_\ell(q)= \varphi(\ell)$, which is 16 only for $\ell = 17$, 
we check that $\Phi_{17}(q) \ne f_{16}(q)$ in $\QQ[q]$, which shows that $\det(D) \neq 0$.
\end{proof} 

%
%
%
\section{Verification of the Main Theorem}

\subsection{The fundamental weight case}\label{S:fundweights}  We can now analyze the question of global irrreducibility for $\Delta(\omega_i)$ for every fundamental weight $\omega_i$.  

\subsection*{Type $A_{n}$ ($n \ge 1$)} All the fundamental weights $\omega_i$, $i=1,2,\dots,n$ are minuscule. Therefore, $\Delta(\omega_i)=L(\omega_i)$ for all $i=1,2,\dots,n$, and ${\ell}\geq 2$. 

\subsection*{Type $B_{n}$ ($n \ge 2$)} The fundamental weight $\omega_n$ is minuscule.  We will verify that $\Delta(\omega_i)$ is reducible for $i=1,2,\dots,n-1$ when ${\ell }=2$. For $B_{n}$, 
$\omega_{1}=\alpha_{0}$, so $\Delta(\omega_{1})$ is reducible when ${\ell}=4$. Now suppose that the statement above holds for $B_{n-1}$. For $2 \le i < n$, restrict to the Levi subgroup of type $B_{n-i+1}$ corresponding to 
$J = \{ \alpha_i, \alpha_{i+1}, \ldots, \alpha_n\}$. Since $\Delta_{J}(\omega_{i})$ is reducible for ${\ell}=4$, it follows that the same holds for $\Delta(\omega_{i})$ by  Theorem \ref{thm:Levireduction}.  

\subsection*{Type $C_{n}$ ($n \ge 3$)} The fundamental weight $\omega_1$ is minuscule. Since $\omega_{2}=\alpha_{0}$, $\Delta(\omega_2)$ is reducible when ${\ell}\mid n$.  For $\omega_i$ with $2 < i < n$, one can 
restrict to the Levi of type $C_{n-i+2}$ corresponding to $J = \{ \alpha_{i-1}, \alpha_i, \ldots, \alpha_n \}$ and apply Theorem \ref{thm:Levireduction} to verify that $\Delta(\omega_i)$ is reducible when ${\ell}\mid n-i+2$.  
Now when $i = n$, restrict to the Levi subgroup of type $C_2 = B_2$ corresponding to $J = \{ \alpha_{n-1}, \alpha_n \}$. One can apply the results for type $B_{2}$ to see that $\Delta_J(\omega_n)$, and thus $\Delta(\omega_{n})$ 
is reducible when ${\ell}=4$.  

\subsection*{Type $D_n$ ($n \ge 4$)} The minuscule fundamental weights are $\omega_1$, $\omega_{n-1}$, and $\omega_n$. For $D_{n}$, $\omega_{2}=\alpha_{0}$, so one can use the same argument as in the case for 
type $B_n$ by restricting to the Levi subgroup of type $D_{n-i+2}$ corresponding to $J = \{ \alpha_{i-1}, \alpha_i, \ldots, \alpha_n \}$ to show that  that $\Delta_J(\omega_i)$ is reducible, and consequently 
$\Delta(\omega_i)$ is reducible when $2 \le i \le n - 2$ and ${\ell} = 4$.

\subsection*{Type $E_6$} By using Theorem~\ref{thm:Levireduction} with $J_{1}=\Delta-\{\alpha_{1}\}$, $J_{2}=\Delta-\{\alpha_{6}\}$ ($D_{5}$ root systems), one can show that the Weyl modules of highest weights 
$\omega_{3}$, $\omega_{4}$ and $\omega_{5}$ are not globally irreducible. The fundamental weights $\omega_{1}$ and $\omega_{6}$ are minuscule. Furthemore, $\omega_{2}=\alpha_{0}$, and 
the Weyl module $\Delta(\omega_{2})$ is not irreducible for ${\ell}=3$.  

\subsection*{Type $E_7$} Set $J_{1}=\Delta-\{\alpha_{7}\}$ ($E_{6}$ root system), 
$J_{2}=\Delta-\{\alpha_{1}\}$ ($D_{6}$ root system). Then by applying Theorem~\ref{thm:Levireduction}, the quantum Weyl module 
with highest weight $\omega_{j}$ is not globally irreducible for $j\neq 1,7$. For the other cases, $\omega_{7}$ is minuscule and $\omega_{1}=\alpha_{0}$. 

\subsection*{Type $E_8$} One can argue as in the prior case, set $J_{1}=\Delta-\{\alpha_{8}\}$ ($E_{7}$ root system), 
$J_{2}=\Delta-\{\alpha_{1}\}$ ($D_{7}$ root system). Then one can conclude that the quantum Weyl module of highest weight $\omega_{j}$ 
for $j\neq 8$ is not globally irreducible. The case of $\omega_8$ is handled in Theorem \ref{thm:shortroot}\eqref{short.E8}.

\subsection*{Type $F_4$} Let $J_{1}=\{\alpha_{1},\alpha_{2},\alpha_{3}\}$ and $J_{2}=\{\alpha_{2},\alpha_{3},\alpha_{4}\}$. By using Theorem~\ref{thm:Levireduction}, the quantum Weyl module 
with highest weight $\omega_{j}$ is not globally irreducible for $j\neq 4$. The case when $\omega_{4}=\alpha_{0}$ is handled in Theorem~\ref{thm:shortroot}\eqref{short.F}.

\subsection*{Type $G_2$} The fundamental weight $\omega_{1}=\alpha_{0}$ so $\Delta(\omega_{1})$ is reducible when ${\ell}=4$. Furthermore, 
$\Delta(\omega_{2})$ is 14-dimensional and not irreducible when ${\ell}=3$. In order to see this one can use a similar analysis as in 
Theorem~\ref{thm:shortroot} with the generators and relations for the 14-dimensional module given in \cite[5.A.4]{Jan2}. The module 
 $\Delta(\omega_{2})$ contains a trivial module if and only if the determinant obtained from these relations is zero, i.e., $[6]_{\zeta}^{2}-[3]_{\zeta}=0$. This occurs when 
 ${\ell}=3$. 

\subsection{} Let $\overline{C}_{\mathbb Z}$ be the bottom alcove, i.e., $\overline{C}_\ZZ =\{\lambda\in X\ : \ \langle \lambda+\rho,\alpha_{0}^{\vee}\rangle \leq {\ell} \}$. If ${\ell}\geq h$ then $0\in \overline{C}_{\mathbb Z}$. For any $\lambda,\mu\in \overline{C}_{\mathbb Z}$ one can define a translation functor 
$T_{\lambda}^{\mu}(-)$. For the basic properties of the translation functor, in the case for algebraic groups, we refer the reader to \cite[II Ch. 7]{Jan}. These properties 
with their proofs directly translate over to the quantum group case. 

\begin{thm} \label{thm:endnodes} Let $\Delta_{\zeta}(\lambda)$ be the quantum Weyl module for $U_{\zeta}({\g})$. 
\begin{enumerate}
\renewcommand{\theenumi}{\alph{enumi}} 
\item \label{end.A} If $\Phi=A_{n}$, $n\geq 2$ and ${\ell}=n+1$ then $\Delta_{\zeta}(\omega_{1}+\omega_{n})$ is reducible. 
\item \label{end.B} If $\Phi=B_{n}$ and ${\ell}=2n+1$ then $\Delta_{\zeta}(\omega_{1}+\omega_{n})$ is reducible. 
\item \label{end.C} If $\Phi=C_{n}$ and ${\ell}=4$ then $\Delta_{\zeta}(\omega_{1}+\omega_{n})$ is reducible. 
\item \label{end.F} If $\Phi=F_{4}$ and ${\ell}=4$ then $\Delta_{\zeta}(\omega_{1}+\omega_{4})$ is reducible. 
\item \label{end.G} If $\Phi=G_{2}$ and ${\ell}=4$ then $\Delta_{\zeta}(\omega_{1}+\omega_{2})$ is reducible. 
\end{enumerate}
\end{thm} 

\begin{proof} For part (\ref{end.A}), if $\Phi=A_{n}$ then $\alpha_{0}=\omega_{1}+\omega_{n}$ and the statement follows from Theorem~\ref{thm:shortroot}(\ref{short.A}). 
An alternative argument in the case when ${\ell}$ is odd can be given using translation functors. Let ${\ell}=n+1$. Then $h=\langle \rho,\alpha_{0}^{\vee} \rangle+1=n+1$. Therefore, $0\in \overline{C}_{\mathbb Z}$. 
Note in this case $\omega_{1}+\omega_{n}=\alpha_{0}$. 

Let $s_{\alpha_{0},{\ell}}$ be the affine reflection (see \cite[II 6.1]{Jan}). Then under the dot action, 
\begin{equation*}
s_{\alpha_{0},{\ell}}\cdot 0 = s_{\alpha_{0}}(\rho)-\rho+{\ell} \alpha_{0} = -\langle \rho, \alpha_{0}^{\vee} \rangle\alpha_{0}+{\ell }\alpha_{0} = (-n+{\ell })\alpha_{0} = \alpha_{0}. 
\end{equation*} 
Consider the hyperplane ${\mathcal H}$ fixed by the affine reflection $s_{\alpha_{0},{\ell}}$ and choose $\mu\in {\mathcal H}\cap \overline{C}_{\mathbb Z}$. 
For this particular $\mu$, one has $L(\mu)=\Delta(\mu)=\nabla(\mu)=T(\mu)$ (where $T(\mu)$ tilting module of highest weight $\mu$). The translated 
module $T_{\mu}^{0}(L(\mu))$ is (i) a tilting module of highest weight $\alpha_{0}$, (ii) has a Weyl filtration with factors ${\mathbb C}$ and $\Delta(\alpha_{0})$, 
and (iii) has socle and head ${\mathbb C}$ with heart (radical/socle) isomorphic to $L(\alpha_{0})$ (cf.~\cite[II 7.19, 7.20]{Jan}). These facts imply that $\Delta(\alpha_{0})$ has composition factors 
$L(\alpha_{0})$ and ${\mathbb C}$, thus $\Delta(\omega_{1}+\omega_{n})$ is reducible. 

\eqref{end.B} The argument used in part (\ref{end.A}) when ${\ell}$ is odd can be used to prove (\ref{end.B}). Assume that ${\ell}=2n+1$ and $\lambda=\omega_{1}+\omega_{n}$. One has $h=2n$ and 
\begin{equation*} 
\langle \omega_{n}+\rho,\alpha_{0}^{\vee} \rangle=\langle \rho,\alpha_{0}^{\vee} \rangle +1 =h=2n < {\ell}
\end{equation*}
Therefore, $\omega_{n}\in \overline{C}_{\mathbb Z}$. Moreover, by direct calculation, 
\begin{equation*} 
s_{\alpha_{0},{\ell}}\cdot \omega_{n}=\omega_{1}+\omega_{n}.
\end{equation*} 
One can apply the same argument as in part (a) to show that $\Delta(\omega_{1}+\omega_{n})$ has composition factors 
$L(\omega_{1}+\omega_{n})$ and $L(\omega_{n})$, thus $\Delta_{\zeta}(\omega_{1}+\omega_{n})$ is reducible. 

For part \eqref{end.C}, we may assume that $n\geq 3$ let $J=\Delta-\{\alpha_{1}\}$ (type $C_{n-1}$). Therefore, the Weyl module 
$\Delta_{J}(\omega_{1}+\omega_{n})$ is reducible for ${\ell}=4$ from the fundamental weight case, and part \eqref{end.C} follows by 
Theorem~\ref{thm:Levireduction}. 

Part \eqref{end.F} follows by the same reasoning as in part \eqref{end.C}, by using $J=\{\alpha_{1},\alpha_{2},\alpha_{3}\}$ (type $B_3$), and part \eqref{end.G}
follows by using $J=\{\alpha_{1}\}$. 
\end{proof}

\subsection{End of the proof of Theorem~\ref{MT}} 
We can now finish the proof of Theorem~\ref{MT}.  The case where $\rank \Phi = 1$ was handled in Example \ref{sl2.eg}.
Suppose $\rank \Phi \ge 2$ and Theorem~\ref{MT} holds for all groups of lower rank, 
and let $\lambda = \sum c_i \omega_i$ with every $c_i \ge 0$. If some $c_i>1$ then one can 
use the case of $U_\zeta(\sl_2)$ from \S\ref{sl2.sec} and Theorem \ref{thm:Levireduction} with $J=\{\alpha_i\}$ to conclude that $\Delta(\lambda)$ is not
globally irreducible. Therefore, we are reduced to the situation where $c_i \in \{ 0, 1 \}$ for all $i$.   

Now if there is a connected and proper subset $J$ of $\Delta$ such that $c_i \ne 0$ for 
at least two indexes $i$ with $\alpha_i \in J$, then we are done by induction and Theorem \ref{thm:Levireduction}. If there are exactly two indices such that $c_i=1$ occurring at the end of the Dynkin diagram such that these nodes are not containing in any connected proper subset $J$ of $\Delta$, then we are in one of the cases handled by Theorem~\ref{thm:endnodes}. Thus we are  reduced to the case when $\lambda$ is a fundamental weight, which was handled in Section~\ref{S:fundweights}.

\bibliographystyle{amsalpha}
\bibliography{e8-minuscule}

%
%
%
%
%
%
%
%
%
%
%

\end{document}